\documentclass[11pt,letterpaper]{article}

\usepackage{enumerate}
\usepackage[pdftex]{color}
\usepackage{amsmath,amsfonts,amssymb}
\usepackage[all]{xy}
\usepackage{graphicx}
\usepackage[margin=2.5cm]{geometry} 
\usepackage{fancyhdr}
\usepackage{hyperref}
\usepackage{multirow}
\usepackage{multicol}
\usepackage[hyperref]{ntheorem}
\usepackage{mathtools}

\def\ser{\leavevmode\raise.585ex\hbox{\small er}}
\theoremstyle{nonumberplain}
\theorembodyfont{\upshape}
\theoremseparator{:}
\newtheorem{definition}{Definition}

\theoremstyle{plain}
\theoremheaderfont{\normalfont\bfseries}
\theorembodyfont{\upshape}
\theoremseparator{:}
\newtheorem{remark}{Remark}[section]
\theoremclass{remark}
\theoremstyle{break}

\theoremstyle{plain} \theoremheaderfont{\normalfont\bfseries}
\theorembodyfont{\slshape} \theoremseparator{.}
\newtheorem{theo}{Theorem}[section]
\theoremclass{theo}
\theoremstyle{break}

\theoremstyle{plain} \theoremheaderfont{\normalfont\bfseries}
\theorembodyfont{\slshape} \theoremseparator{.}

\theoremstyle{plain} \theoremheaderfont{\normalfont\bfseries}
\theorembodyfont{\slshape} \theoremseparator{.}
\newtheorem{lemma}[theo]{Lemma}
\theoremstyle{plain} \theoremheaderfont{\normalfont\bfseries}
\theorembodyfont{\slshape} \theoremseparator{.}
\newtheorem{prop}[theo]{Proposition}
\newenvironment{proof}[1][Proof]{\textbf{#1.} }{\ \rule{0.5em}{0.5em}}

\DeclareMathOperator{\Fib}{\mathsf{Fib}}
\DeclareMathOperator{\FI}{\mathsf{FI}}
\DeclareMathOperator{\FB}{\mathsf{FB}}

\DeclareMathOperator{\Mod}{Mod}
\DeclareMathOperator{\PMod}{PMod}

\DeclareMathOperator{\Diff}{Diff}
\DeclareMathOperator{\PDiff}{PDiff}

\DeclareMathOperator{\Hom}{Hom}

\DeclareMathOperator{\Ind}{Ind}

\newcommand{\BDiff}{B\Diff}
\newcommand{\BPDiff}{B\PDiff}

\newcommand{\PConf}{\text{PConf}}

\title{ Linear representation stable bounds for the integral cohomology of  pure mapping class groups } 
\author{ R. Jim\'enez Rolland \footnote{This paper has been completed with the financial support from  PAPIIT DGAPA-UNAM grant  IA104010.}} 
\date{}

\begin{document}

\maketitle
\begin{abstract} 


In this paper we study the integral cohomology  of pure mapping class groups of surfaces, and other related groups and spaces, as $\FI$-modules. We use recent results from Church, Miller, Nagpal and Reinhold to obtain explicit linear bounds for their presentation degree and  to give an inductive  description of these $\FI$-modules. Furthermore, we establish new results on  representation stability, in the sense of Church and Farb,  for the rational cohomology of pure mapping class groups of non-orientable surfaces.


\end{abstract}


\section{Introduction}

Let $	\Sigma$ be a compact connected surface with (possibly empty) boundary and let $P =  \{p_1,\ldots,p_n\}$ be a set of $n$ distinct points in the interior of $\Sigma$, we call these points {\it punctures}. Let $\Diff(\Sigma, P)$  be the group of all, orientation preserving if $\Sigma$ is orientable, homeomorphisms $h: \Sigma\rightarrow\Sigma$ such that $h(P)=P$ and $h$ restricts to the  identity on the boundary of $\Sigma$. The {\it mapping class group} $\Mod^n(\Sigma)$ is the group of isotopy classes of elements of $\Diff(\Sigma, P)$. The {\it pure mapping class group} $\PMod^n(\Sigma)$ is the subgroup of $\Mod^n(\Sigma)$ that consists of the isotopy classes of homeomorphisms fixing each puncture. If $P=\varnothing$, we write $\Mod(\Sigma)$.

In this paper we study the $\FI$-module structure of the integral cohomology of  the pure mapping class groups of orientable and non-orientable surfaces.  A compact surface of genus $g$ with $r$ boundary components  and $n$ punctures will be denoted by $\Sigma_{g,r}^n$ if it is orientable, or by $N_{g,r}^n$ if it is non-orientable.

In  \cite{CMNR} the authors study the {\it stable degree} and the {\it local degree} of an $\FI$-module and show that these invariants are easier to control in spectral sequence arguments than the generation and presentation degree. We review these notions and results in Section \ref{FIMod} below. In particular, they obtain  bounds on the stable degree and the local degree for the cohomology of configuration spaces of connected manifolds (see Lemma \ref{CONF} below).  

We recall in Section \ref{MCG} how the Birman exact sequence and certain fibrations allow us to relate the cohomology of pure mapping class groups with the cohomology of configuration spaces through a  spectral sequence.  In contrast with our previous work  in \cite{JimenezRollandMCGasFI} and  \cite{JimenezRollandRepStability}, we work over the integers and include non-orientable surfaces.  We study  here how the bounds on the  stable and local degree of  certain $\FI$-modules $V$ that we call  $\FI[\mathsf{G}]$-modules,  provide bounds for the $\FI$-module  given by the  cohomology  of a group $G$ with coefficients in $V$
(see Proposition \ref{LEMMA}). We combine this with the results from \cite{CMNR} to prove in  Theorem \ref{MAINMCG}  that  for every $k\geq0$ and any abelian group $A$, the $\FI$-module  $H^k\big( \PMod^\bullet(\Sigma);A\big)$ is finitely presented. Furthermore, in Theorem \ref{MAINMCG} we give explicit bounds for its presentation degree, stable degree and local degree, see also Theorem \ref{BDRY}.  
As a consequence from Propositions \ref{POLY} and \ref{INDUCT}, we obtain the following result.


\begin{theo}\label{MAIN1}  Let  $\Sigma$ be a surface such that:
\begin{itemize}
 \item[] $\Sigma=S^2, \Sigma_1^1$ or  $\Sigma=\Sigma_{g,r}$ with $2g+r>2$, if  the surface is orientable,  or 
 
 \item[] $\Sigma=\mathbb{R}P^2, \mathbb{K}$ or $\Sigma=N_{g,r}$  with $g\geq 3$ and $r\geq 0$ in the non-orientable case. 
 \end{itemize} 
 Let $k\geq 0$ and $\lambda=1$ if $\Sigma$ is orientable and $\lambda=0$ if $\Sigma$ is non-orientable. 
 
\begin{itemize}
\item[(a)] {\bf (Polynomial Betti numbers growth).} If $\mathbb{F}$ is a field, then there are polynomials $p^\Sigma_{k,\mathbb{F}}$ of degree at most $2k$ such that
$$\dim_{\mathbb{F}} H^k\big( \PMod^n(\Sigma);\mathbb{F}\big)= p^\Sigma_{k,\mathbb{F}}(n)$$
if  $n >\max(-1,16k-4\lambda-2)$  (and for all $n\geq 0$ if the surface $\Sigma$ has nonempty boundary).

\item[(b)] {\bf (Inductive description). }The natural map
$$\Ind^{S_n}_{S_{n-1}}H^k \big( \PMod^{n-1}(\Sigma);\mathbb{Z}\big)\rightarrow H^k \big( \PMod^{n}(\Sigma);\mathbb{Z}\big)$$
is surjective for $n >\max(0,18k-4\lambda-1)$  (for $n>2k$ if $\partial\Sigma\neq \varnothing$) and, for $n >\max(0,34k-8\lambda-2)$  (for $n>2k$ if $\partial\Sigma\neq \varnothing$), the kernel of this map is the image of the difference of the two natural maps\footnote{See Proposition \ref{INDUCT} and Remark \ref{TWOMAPS} for details.}

$$\Ind^{S_n}_{S_{n-2}}H^k \big( \PMod^{n-2}(\Sigma);\mathbb{Z}\big)\rightrightarrows \Ind^{S_n}_{S_{n-1}}H^k \big( \PMod^{n-1}(\Sigma);\mathbb{Z}\big).$$
\end{itemize}

\end{theo}


Theorem \ref{MAIN1}(a) is known to be true for all $n\geq 0$  for pure mapping class groups of orientable surfaces with boundary   (\cite[Theorems 1.2]{JimenezRollandMCGasFI})  and in the stable range  for orientable closed surfaces (\cite[Theorem 1.5]{JimenezRollandMCGasFI}).  Furthermore,  in \cite[Theorem 1.1]{BOTIL} B{\"o}digheimer and Tillmann gave a decomposition of the classifying space associated with the stable pure mapping class groups of orientable surfaces with boundary $\PMod^n(\Sigma_{\infty,r})$. Together with Harer's homological stability theorem  \cite{HarerStability} this implies that for any field $\mathbb{F}$, in cohomological degrees $*\leq g/2$, 
 $$H^*(\PMod ^n(\Sigma_{g,r});\mathbb{F}) \cong H^* (\Mod(\Sigma_{g,r});\mathbb{F}) \otimes\mathbb{F}[x_1,\ldots,x_n],$$
\noindent  where each $x_i$ has degree $2$. The action of the symmetric group $S_n$  on the left hand side corresponds to  the action of $S_n$ on the polynomial ring in $n$ variables by permutation of the variables $x_i$. 
Therefore, for an orientable surface $\Sigma$ with non-empty boundary,    it follows that for $*\leq g/2$ the polynomials $p^{\Sigma}_{*,\mathbb{F}}(n)$ from Theorem \ref{MAINMCG}(a) are given by 
\[
p^{\Sigma}_{*,\mathbb{F}}(n)= \dim_\mathbb{F} H^*(\PMod ^n(\Sigma_{g,r});\mathbb{F}) =   \begin{dcases}
        \dim_\mathbb{F} H^*(\Mod(\Sigma_{g,r});\mathbb{F}) & \text{ if  }*=2k+1, \\
        {{n+k-1}\choose{n-1}}+\dim_\mathbb{F} H^*(\Mod(\Sigma_{g,r});\mathbb{F})  & \text{ if }*=2k,\\
    \end{dcases}
\]
of degree $0$ and $k$, respectively.


Theorem \ref{MAIN1}(b)   gives an  inductive description of the $\FI$-module in terms of a explicit finite amount of data; it is equivalent to the statement that the   $\FI$-module $H^k(\PMod^{\bullet}(\Sigma);\mathbb{Z})$ is {\it centrally stable}  at $>\max(0,34k-8\lambda-2)$, in the sense of  Putman's  \cite{PUTCentral} original notion of {\it central stability}  (see Section \ref{GENSEC}). 

 \begin{remark} In \cite{PUTStab}, a preprint version of   \cite{PUTCentral}, Putman proved that the homology of pure mapping class groups of manifolds with boundary satisfies  central stability. His   result \cite[Theorem D]{PUTStab} includes orientable and  non-orientable manifolds, not necessarily of finite type, with coefficients in a field (of characteristic bounded below or zero characteristic) and obtains an exponential stable range. In contrast, our approach for  dimension $2$  works over $\mathbb{Z}$ and gives linear bounds for the stable range. It  would be interesting to see whether the central stability techniques used in \cite{PUTStab} together with results from \cite{CMNR} could be used to obtain  better bounds and include mapping class groups for general manifolds.
  \end{remark}

Furthermore, from \cite[Theorem 1.5, 1.13]{CEF}  (see also \cite[Proposition 3.9]{CEFPointCounting}) it follows that Theorem \ref{MAINMCG} also implies that the rational cohomology of  pure mapping class group of surfaces satisfies {\it representation stability} and have characters that are eventually polynomial.  From Proposition \ref{STABLERAN} and the upper bounds for stable degree and the local degree  obtained in Theorem \ref{MAINMCG},  an explicit  stable range is obtained.


\begin{theo}\label{REPSTABNOR}  
Let  $k\geq 0$ and consider  $\Sigma$  a surface as in Theorem \ref{MAIN1}. Let
$$N= 
    \begin{cases}
       4k& \text{if  }\partial\Sigma\neq\varnothing,\\
       \max(0,20k-4\lambda-1)& \text{otherwise, with  $\lambda=1$ if $\Sigma$ is orientable and $\lambda=0$ if not}.
   \end{cases}$$


\begin{itemize}
\item[{\bf i)}] {\bf (Representation stability over $\mathbb{Q}$).} The decomposition of $H^k\big( \PMod^n(\Sigma);\mathbb{Q}\big)$ into irreducible $S_n$-representations stabilizes in the sense of uniform representation stability (as defined in \cite{CF})  for $n\geq N$. 

\item[{\bf ii)}]{\bf (Polynomial characters).} For $n\geq N$, the character $\chi_n$  of $H^k\big( \PMod^n(\Sigma);\mathbb{Q}\big)$ is of the form:
$\chi_{n}=Q_k(Z_1,Z_2,\ldots,Z_r),$
where deg$(\chi_n)\leq 2k$  and $Q_k\in\mathbb{Q}[Z_1,Z_2,\ldots]$  is a unique polynomial  in the class functions $$Z_{\ell}(\sigma):=\# \text{ cycles of length  $\ell$ in }\sigma, \text{\hspace{5mm} for  any }\sigma\in S_n$$  that is independent of $n$. 

 \item[{\bf iii)}] {\bf (Stable inner products).} If $Q\in\mathbb{Q}[Z_1,Z_2,\ldots]$  is any character polynomial, then the inner product  $\langle \chi_{n},Q\rangle_{S_n}$ is  independent of $n$ once $n\geq\max(N,2k+\deg(Q))$.
 
\item[{\bf iv)}] {\bf (``Twisted'' homological stability over $\mathbb{Q}$).}  Let $\{W_n\}$ be a sequence of finite dimensional $S_n$-representations, then the dimension $\dim_{\mathbb{Q}}\big(H^i(\Mod^n(\Sigma);W_n)\big)$ is constant for $n\geq N$. In particular, 
the sequence $\{H^k \big( \Mod^n(\Sigma);\mathbb{Q}\big)\}$ satisfies rational homological stability for $n\geq N$.
\end{itemize}
\end{theo}

Theorem \ref{REPSTABNOR}  above recovers \cite[Theorem 1.1]{JimenezRollandMCGasFI} for the pure mapping class group of  orientable surfaces and  gives new representation stability results for the rational cohomology of pure mapping class groups of  non-orientable surfaces and low genus cases that were not worked out explicitly in  \cite{JimenezRollandMCGasFI} and \cite{JimenezRollandRepStability}. 
 Using the theory of weight and stability degree from \cite{CEF} better bounds can be computed for the stable range. However, if the surface $\Sigma$ has nonempty boundary, from Proposition \ref{STABLERAN} the stable range can be improved to  $N=4k$, which recovers the stable range obtained in  \cite[Theorem 1.1]{JimenezRollandMCGasFI} for orientable surfaces with boundary.   

Theorem  \ref{REPSTABNOR}(iv) recovers rational homological stability ``by punctures'' for mapping class groups of surfaces, this was previously obtained in \cite[Propositon 1.5] {HatcherWahl}, with a better linear stable range, for mapping class groups of connected manifolds  of dimension $\geq 2$ (see also \cite[Theorems 1.3 \& 1.4]{HANDBURY}).\bigskip




\noindent{\large \bf Known results for pure mapping class groups of non-orientable surfaces.}
Wahl proved in  \cite[Theorem A]{WAHLNONORIENT}  that for $g\geq 4k+3$, the homology  groups $H_k(\Mod(N_{g,r});\mathbb{Z})$ are independent of $g$ and $r$. This is  the non-orientable  version of Harer's stability theorem for mapping class group of orientable surfaces \cite{HarerStability}.  Homological stability for mapping class groups of non-orientable surfaces with marked points was obtained in  \cite[Theorem 1.1]{HANDBURY}: 
\begin{center}
{\it for fixed $n\geq 1$ the homology  groups $H_k(\Mod^n(N_{g,r});\mathbb{Z})$ and $H_k(\PMod^n(N_{g,r});\mathbb{Z})$

 are independent of $g$ and $r$,  when $g\geq 4k+3$.}
\end{center}

 Wahl and Handbury also study the stable  pure mapping class groups of  non orientable surfaces with boundary $\PMod^n(N_{\infty,r})$. 
In particular, \cite[Theorem 1.2]{HANDBURY} shows that there exists a homology isomorphism $B\PMod^n(N_{\infty,r})\rightarrow B\PMod(N_{\infty,r})\times B(O(2))^n.$ This isomorphism could be use to explicitly compute, for $*\geq 4k+3$ the polynomials $p^{\Sigma}_{*,\mathbb{Z}_2}(n)$ from Theorem \ref{MAINMCG}(a) for a non-orientable surface $\Sigma$ with non-empty boundary.
On the other hand, Korkmaz \cite{KOR}  and Stukow \cite{STUKOW} computed the first homology group for the mapping class group of a non-orientable surfaces. In particular, they proved that $$H_1(\PMod^n(N_g);\mathbb{Z})\cong (\mathbb{Z}_2)^{n+1}\text{ if }g\geq 7,$$ see for instance \cite[Theorem 5.11]{KOR}. This computation shows that that homological stability fails integrally for the pure mapping class group with respect to the number of marked points,  and suggest a representation stability phenomenon.

For low genus, descriptions of the mod $2$ cohomology ring have been obtained for  the mapping class group $\Mod^n(\mathbb{K})$ of the Klein  bottle $\mathbb{K}$ with marked points (\cite{KLEIN}), and for the mapping class group  $\Mod^n(\mathbb{R}P^2)$  of the projective plane $\mathbb{R}P^2$ with marked points (\cite{PROJECTIVE}). Similar techniques should provide computations for the the mod $2$ cohomology ring for $\PMod^n(\mathbb{K})$ and $\PMod^n(\mathbb{R}P^2)$. \bigskip

\noindent{\large \bf Further results.}
With the same techniques, we obtain bounds for the  presentation degree of the $\FI$-module  given by the integral cohomology of the {\it hyperelliptic mapping class group with marked points} $\Delta_g^n$  and representations stability for its rational cohomology groups in  Theorem \ref{HYPERMCG}. For more details, see Section \ref{HYPERSEC}.\\

Finally we obtain Theorem \ref{DIFFEO}, an integral version of \cite[Theorem 6.6]{JimenezRollandMCGasFI} for the cohomology of classifying spaces $B\PDiff^n(M)$ of pure diffeomorphisms groups of a manifold $M$ of dimension $\geq 3$. Theorem \ref{DIFFEO} also includes the case of $M$ being non-orientable.   As a consequence we have the following result:

\begin{theo}\label{DIFFEOCOR}  Let $k\geq 0$ and  $M$ be a connected real manifold (orientable or non-orientable) of dimension $d\geq 3$. Let $\lambda=1$ if $M$ is orientable and $\lambda=0$ if $M$ is non-orientable.
\begin{itemize}
\item[(a)]  Let $\mathbb{F}$ be a field. If $H^k\big( \BPDiff^n(M);\mathbb{F}\big)$ has  finite dimension for every $n\geq 1$, then there are polynomials $p^M_{k,\mathbb{F}}$ of degree at most $k$ such that
$$\dim_{\mathbb{F}} H^k\big( \BPDiff^n(M);\mathbb{F}\big)= p^M_{k,\mathbb{F}}(n)\text{ \hspace{0.5cm} if }n >\max(-1,8k-2\lambda-2).$$
\item[(b)] The natural map
$$\Ind^{S_n}_{S_{n-1}}H^k \big( \BPDiff^{n-1}(M);\mathbb{Z}\big)\longrightarrow H^k \big(\BPDiff^n(M);\mathbb{Z}\big)$$
is surjective for $n >\max(0,9k-2\lambda-1)$ and for $n >\max(0,17k-4\lambda-2)$ the kernel of this map is the image of the difference of the two natural maps
$$\Ind^{S_n}_{S_{n-2}}H^k \big( \BPDiff^{n-2}(M);\mathbb{Z}\big)\rightrightarrows \Ind^{S_n}_{S_{n-1}}H^k \big(\BPDiff^{n-1}(M);\mathbb{Z}\big).$$
\end{itemize}
\end{theo}

\begin{remark}This paper started from  Miller and Wilson's Remark A.14 in \cite{MIWILSON}. They suggest there  that  bounds for generation and relation degree of the $\FI$-module $H^k\big(B\PDiff^{\bullet}(M);\mathbb{Z}\big)$ could  be obtained using \cite[Cor. A.4]{MIWILSON}. The advantage of using the techniques from \cite{CMNR} is that linear bounds can be obtained.
\end{remark}


\noindent{\bf Acknowledgements.} I am grateful to  Marco Boggi, Cristhian Hidber,  Jeremy Miller, Jenny Wilson and Miguel Xicot\'encatl   for useful discussions and generous explanations. I also want to thank 
Jes\'us Hern\'andez Hern\'andez and  Blazej Szepietowski for pointing out key references on mapping class groups of non-orientable surfaces.

\section{ Preliminaries on $\FI$-modules }\label{FIMod}

We use the framework of $\FI$-modules to study stability properties of sequences of symmetric group representations. \\

\noindent {\bf Notation:} An {\it  $\FI$-module}  (resp. {\it  $\FB$-module}) is a functor $V$ from the category of finite sets and injections (resp. and bijections) to the category of $\mathbb{Z}-$modules $\Mod_{\mathbb{Z}}$. Given a finite set $T$  we write $V_T$ for $V(T)$ and for every $n\in \mathbb{N}$ we write $V_n$ for $V_{[n]}=V([n])$ where $[n]:=\{1,2,\ldots n\}$. The category of $\FI$-modules $\Mod_{\FI}$ is abelian.

\begin{definition} Let $\mathbb{F}$ be any ring. An $\FI$-module $V$ is an {\it $\FI$-module over $\mathbb{F}$}  if the functor $V : \Mod_{\FI}\rightarrow \Mod_{\mathbb{Z}}$ factors through $\Mod_{\mathbb{F}}\rightarrow \Mod_{\mathbb{Z}}$.

The forgetful functor $\mathcal{F}:\Mod_{\FI} \rightarrow\Mod_{\FB}$ has  left adjoint that we denote  by $\mathcal{I} : \Mod_{\FB}\rightarrow\Mod_{\FI}$. The $\FI$-modules of the form $\mathcal{I}(V)$ are called {\it induced}.  The $\FI$-modules that  admit a finite length filtration where  the quotients are induced modules are named {\it semi-induced}.
\end{definition}



\begin{definition}An $\FI$-module $V$ is generated by a  set $S\subseteq \sqcup_{n\geq 0} V_n$ if $V$ is the smallest FI-submodule containing S.  If $V$ is generated by some finite set, we say that $V$ is {\it finitely generated}.
\end{definition}

Finite generation of $\FI$-modules is preserved when taking subquotients, extensions and tensor products.
 Furthermore  $\FI$-modules over Noetherian rings are locally Noetherian  (\cite[Theorem A]{CEFN}). 
 Therefore  an $\FI$-module is finitely generated if and only if it is {\it finitely presented}.

\subsection{Generation and presentation degree}\label{GENSEC}

Consider the  functor $\mathcal{U}:\Mod_{\FB}\rightarrow\Mod_{\FI}$  that upgrades each $\FB$-module to an $\FI$-module by declaring that all injections which are not bijections act as the zero map. The left adjoint is called {\it $\FI$-homology} and  it is a right-exact functor denoted by $H_0^{\FI}$. Its $i$th left-derived functor  is denoted by $H_i^{\FI}$, and will be often regarded as an $\FI$-module by post-composing  with $\mathcal{U}$.

The {\it degree} of a non-negatively graded abelian group $M$ is $\deg M =\min\{d \geq -1: M_n = 0\text{ for }n > d\}$. In particular, an $\FI$-module or $\FB$-module $V$  gives a non-negatively graded abelian group by evaluating on the sets $[n]$ and  $\deg V$ can be considered. 

\begin{definition}  For an $\FI$-module or
$\FB$-module $V$ let  $t_i(V) := \deg H^{\FI}_i(V)$.  The {\it generation degree}  of $V$ is  $t_0 (V)$ and $prd(V):=\max(t_0(V),t_1(V))$  is called the {\it presentation degree} of $V$. An $\FI$-module $V$ is {\it presented in finite degrees} if $t_0(V) <\infty$ and $t_1(V)<\infty$.
\end{definition}

  \begin{remark} A finitely generated $\FI$-module is presented in finite degrees. However,  being `finitely generated'  is a stronger condition than being presented in finite degree.   \end{remark}

The category of $\FI$-modules presented in finite degrees is abelian. Furthermore $\FI$-modules admit the following uniform description:
  
  \begin{theo}[Theorem C \cite{CE}]\label{UNIFORM} Let $V$ an $\FI$-module with presentation degree $prd(V)=N$. Then for any finite set $T$, 
\begin{equation}\label{CENTRAL} V_T=\text{colim}_{\substack{S\subset T\\ |S|\leq N}} V_S,\end{equation}
and $N$ is the smallest integer such that (\ref{CENTRAL}) holds for all finite sets.
\end{theo}

If $V$ is a finitely generated $\FI$-module, then  Theorem \ref{UNIFORM} above gives a uniform description of an $\FI$-module in terms of a explicit finite amount of data.  
The condition  (\ref{CENTRAL}) in Theorem \ref{UNIFORM} can be viewed as a reformulation of Putman's {\it central
stability} condition \cite[Section 1]{PUTCentral}. In particular, if an $\FI$-module $V$ is presented in finite degrees with  $prd(V)\leq N$ then the sequence $\{V_n\}_n$ is {\it centrally stable at  $n> N$ }(see \cite{CE}  and also \cite[Theorem 3.2 ]{GL}). Putman's {\it central stability} condition \cite[Section 1]{PUTCentral}  is equivalent to the following inductive description (see for example \cite[Prop 6.2]{PATZT}):

\begin{prop}[Prop 2.4 \cite{CMNR}]\label{INDUCT}  Let $V$ be an $\FI$-module.
\begin{itemize}
\item[(1)] Then $t_0(V) \leq d$ if and only if $\Ind^{S_n}_{S_{n-1}} V_{n-1}\rightarrow V_n$ is surjective for $n > d$.
\item[(2)] Then $t_1(V) \leq r$ if and only if the kernel of $\Ind^{S_n}_{S_{n-1}} V_{n-1}\rightarrow V_n$ is the image of the difference of the two natural maps $\Ind^{S_n}_{S_{n-2}} V_{n-2}\rightrightarrows \Ind^{S_n}_{S_{n-1}} V_{n-1}$  for $n > r$. 
\end{itemize}
\end{prop}


\begin{remark}\label{TWOMAPS} Observe that $\Ind^{S_n}_{S_{n-p}} V_{n-p}\cong \bigoplus_{f:[p]\hookrightarrow [n]}V_{[n]\setminus im(f)}$ for $p\geq 1$. The two natural maps  $\Ind^{S_n}_{S_{n-2}} V_{n-2}\rightrightarrows \Ind^{S_n}_{S_{n-1}} V_{n-1}$ in Proposition \ref{INDUCT} (2) are given by 
$$ d_i: \bigoplus_{f:[2]\hookrightarrow [n]}V_{[n]\setminus im(f)}\rightarrow  \bigoplus_{\bar{f}=f|[2]\setminus\{i\}}V_{[n]\setminus im(\bar{f})}$$
defined by forgetting the element $i$ from the domain $[2]$ of the injective map $f:[2]\hookrightarrow [n]$ and the maps $V_{[n]\setminus im(f)}\rightarrow V_{[n]\setminus im(\bar{f})}$ induced by the inclusions of sets $\big([n]\setminus im(f)\big)\hookrightarrow \big([n]\setminus im(\bar{f})\big)$. 
\end{remark}


Furthermore, if $V$ is an $\FI$-module over a field of characteristic zero $\mathbb{F}$, then $V$  is finitely generated if and only if the sequence $\{ V_n\}_n$ of $S_n$-representations is {\it representation stable} in the sense of \cite{CF} and each $V_n$ is finite-dimensional (see \cite[Theorem 1.13]{CEF}), with stable range $\leq 2(prd(V)+1)+1$ (see for example \cite[Theorem E]{PUTCentral}).  The notions of {\it weight} and {\it stability degree} introduced  in \cite{CEF} can be used to obtain better explicit stable ranges for representation stability over fields of characteristic zero.


\subsection{Stable and local degree}

In  \cite{CMNR} the authors study the {\it stable degree}  and  the {\it local degree}  of an $\FI$-module and show that they are easier to control in spectral sequence arguments than the generation and presentation degree. The precise definitions of these invariants are stated in \cite[Def. 2.8 \& 2.12]{CMNR}. We recall here the key properties that will be needed to obtain our main results. Proofs of these statements can be found in \cite[Section 3]{CMNR}.

The category $\FI$ has a proper self-embedding $\sqcup^{\star}:\FI\rightarrow \FI$ given by $S\mapsto S\sqcup\{\star\}$ and a morphism $f:S\rightarrow T$ extends to $f_\star:S\sqcup\{\star\}\rightarrow T\sqcup\{\star\}$ with $f_*(\star)=\star$.
The shift functor $\mathcal{S}:\Mod_{\FI}\rightarrow \Mod_{\FI}$ is given by $V\mapsto V\circ\sqcup^{\star}$.  We denote by $\mathcal{S}^n$  the $n$-fold iterate of $\mathcal{S}$. 

The {\it local degree} $h^{\text{max}}(V)$ of an $\FI$-module $V$   quantifies how much  $V$ needs to be shifted so that $\mathcal{S}^n V$ is semi-induced:

\begin{prop}[Cor 2.13\cite{CMNR}]\label{LOCALDEG}
Let $V$be an $\FI$-module presented in finite degrees. Then $\mathcal{S}^n V$ is semi-induced if and only if $n>h^{\text{max}}(V)$.
\end{prop}

The natural transformation $\text{id}_{\FI}\rightarrow \sqcup^{\star}$ induces a natural transformation $\iota:\text{id}_{\Mod_{\FI}}\rightarrow\mathcal{S}$.  This is a natural map of $\FI$-modules $\iota: V\mapsto \mathcal{S}V$  which, for every finite set $S$, takes $V_S$ to $(\mathcal{S} V)_S=V_{S\cup\{\star\}}$ via the map corresponding to the inclusion $\iota_S: S\hookrightarrow S\cup\{\star\}$.
The derivative functor $\Delta:\Mod_{\FI}\rightarrow\Mod_{\FI}$ takes an $\FI$-module $V$ to  the cokernel $\Delta V:=\text{coker}(V\xrightarrow\iota \mathcal{S} V)$. We denote by $\Delta^n$ is the $n$-fold iterate of $\Delta$.  An $\FI$-module $V$ is {\it torsion} if all its elements are torsion.  An element $x\in V(S)$ is  {\it torsion} if there is an injection $f:S\rightarrow T$ such that $f_*(x)=0$.

\begin{definition}  
The {\it stable degree}  of an $\FI$-module $V$ is  $\delta(V):=\min \{n\geq -1: \Delta^{n+1} V\text{ is torsion}\}$.\end{definition}

 

For a finitely generated $\FI$-module $V$ over a field $\mathbb{F}$, the  dimensions $\dim_{\mathbb{F}} V_n$  are eventually equal to a polynomial in $n$ (\cite[Theorem B]{CEFN}). The   stable degree  $\delta(V)$ of a finitely generated $\FI$-module $V$  corresponds to the degree of this polynomial and the  local degree $h^{\text{max}}(V)$ controls when these dimensions become equal to this polynomial.

\begin{prop}[Prop 2.14 \cite{CMNR}]\label{POLY} Suppose $\mathbb{F}$ is a field, and let $V$ be an $\FI$-module over $\mathbb{F}$ which is presented in finite degrees and with $V_n$ finite dimensional for all $n$. Then there exists an integer-valued polynomial $p \in\mathbb{Q}[X]$ of degree $\delta(V)$ such that $dim_\mathbb{F} V_n = p(n)$ for $n > h^{\text{max}}(V)$.
\end{prop}

It turns out that together, the stable degree and the local degree   behave well under taking kernels and cokernels and in finite filtrations.
 
 \begin{prop}[Prop 3.2 \& 3.3 \cite{CMNR}]\label{KERCOKER}\ 
 \begin{itemize}
 \item[(1)] Suppose the $\FI$-module $V$ has a finite filtration $V = F_ 0\supset\ldots\supset F_k=0$ and let $N_i = F_i/F_{i+1}$. Then $\delta(V) = \max_i\delta(N_i)$ and $h^{\text{max}}(V) \leq \max_i h^{\text{max}}(N_i)$.
 \item[(2)] Let $f:V\rightarrow W$ be a map of $\FI$-modules presented in  finite degrees. Then we have the following:
  $$\delta(\ker f)\leq \delta(V);\text{\hspace{0.5cm}}\delta(\text{coker} f)\leq \delta(W);$$ 
 $$h^{\text{max}}(\ker f), h^{\text{max}}(\text{coker} f)\leq \max\big(2\delta(V)-2, h^{\text{max}}(V),h^{\text{max}}(W)\big).$$ 
 \end{itemize}
 \end{prop}

Furthermore, the stable and local degrees are easier to control in spectral sequence arguments. In particular they allow us to obtain linear stable ranges in the  ``Type A spectral sequence arguments''.

\begin{prop}[Prop 4.1 \cite{CMNR}] \label{SPECTRAL} Let $E_r^{p,q}$ be a cohomologically graded first quadrant spectral sequence of $\FI$-modules converging to $M^{p+q}$. Suppose that for some page $d$, the $\FI$-modules $E_d^{p,q}$ are presented in finite degrees, and set $D_k=\max_{p+q=k} \delta(E_d^{p,q})$ and $\eta_k=\max_{p+q=k} h^{\text{max}}(E_d^{p,q})$. Then we have the following:
\begin{itemize}
\item[(1)] $\delta(M^k)\leq D_k$
\item[(2)] $h^{\text{max}}(M^k)\leq \max \big(\max_{\ell\leq k+s-d} \eta_\ell, \max_{\ell\leq 2k-d+1} (2D_\ell-2)\big)$
\end{itemize}
where $s=\max(k+2,d)$.
\end{prop}

Finally, the stable degree and the local degree control the presentation degree of an $\FI$-module (and viceversa \cite[Prop 3.1 (1)\&(2)]{CMNR}):

   \begin{prop}[Prop 3.1 \cite{CMNR}]\label{BOUNDS}  Let $V$ be an $\FI$-module presented in finite degree. Then we have the following: $$t_0(V)\leq \delta(V)+h^{\text{max}}(V)+1\text{ \hspace{0.5cm} and  \hspace{0.5cm} }t_1(V)\leq \delta(V)+2h^{\text{max}}(V)+2.$$
 \end{prop}

\subsubsection*{$\FI\#$-modules}

Let $\FI\#$ denote the category with objects finite based sets and with  morphisms given by maps of based sets that are injective away from the basepoints:  the preimage of all elements except possibly the base point have cardinality at most one. This category is isomorphic to the original one described in \cite[Def 4.1.1]{CEF}.  

An {\it $\FI\#$-module} is a functor from $\FI\#$ to $\Mod_\mathbb{Z}$.  An $\FI\#$-module simultaneously carries an $\FI$- and an $\FI^{op}$-module structure in a compatible way.  Church--Ellenberg--Farb \cite[Theorem 4.1.5]{CEF} proved that the restriction of an $\FI\#$-module to $\FI$ is an induced $\FI$-module. Hence, for $\FI\#$-modules we have better control of the presentation degree.

\begin{prop}[Cor 2.13, Prop 3.1 \cite{CMNR}]\label{FISHARP}
If $V$ is a semi-induced $\FI$-module, then $h^{\text{max}}(V)=-1$.  Therefore,  $\delta (V)=t_0(V)$ and $t_1(V)\leq \delta (V)$.
\end{prop}

For a finitely generated $\FI$-module $V$ defined over $\mathbb{Q}$, we know that the sequence $\{ V_n\}_n$ satisfies representation stability. Moreover we can  obtain upper bounds for the stable range in terms of the local and stable degree of $V$.

\begin{prop}\label{STABLERAN} Let $V$ be a finitely generated $\FI$-module defined over $\mathbb{Q}$. Then the sequence $\{V_n\}_n$ satisfies representation stability for $n\geq 2\delta(V)+ h^{\text{max}}(V)+1$.  If $V$ is an $\FI\#$-module defined over $\mathbb{Q}$,  then the sequence $\{V_n\}_n$ satisfies representation stability for $n\geq 2\delta(V)$.
\end{prop}
\begin{proof}
Over $\mathbb{Q}$,  semi-induced modules are the same as $\FI\#$-modules (since $\FI\#$-modules are semi-simple).
By Proposition \ref{LOCALDEG}, if $n>h^{\text{max}}(V)$, then  $\mathcal{S}^n V$ is semi-induced.   Then for $N=h^{\text{max}}(V)+1$, the shifted $\FI$-module $\mathcal{S}^N V$  is an $\FI\#$-module defined over $\mathbb{Q}$ and  by Proposition \ref{FISHARP} it has  generation degree $t_0(\mathcal{S}^N V)=\delta(\mathcal{S}^N V)=\delta(V)$ (see also \cite[Prop 2.9(b)]{CMNR}). Therefore, by \cite[Corollary 4.1.8]{CEF} the sequence $\{(\mathcal{S}^N V)_k\}_k$ is representation stable for $k\geq 2\delta(V)$. Since $\big(\mathcal{S}^N V\big)_k=V_{N+k}$, this means that the sequence $\{V_n\}_n$   satisfies representation stability  for  $n\geq 2\delta(V)+N=2\delta(V)+h^{\text{max}}(V)+1$.
\end{proof}




\section { $\FI\mathsf{[G]}$-modules}

In \cite[Section 5]{JimenezRollandMCGasFI} we introduced the notion on an $\FI\mathsf{[G]}$-module in order to incorporate the action of a group $G$ on our sequences of $S_n$-representations. These  $\FI\mathsf{[G]}$-modules  allow us to construct new $\FI$-modules by taking cohomology with twisted coefficients and are key in the spectral sequence arguments below.



\begin{definition} Let $G$ be a group. An  {\it $\FI\mathsf{[G]}$-module $V$} is a functor from the category {$\FI$} to the category $\Mod_{\mathbb{Z}\mathsf{[G]}}$ of $G$-modules over $\mathbb{Z}$.   By forgetting the $G$-action we get an $\FI$-module and all the notions of degree from the previous section can be considered.
We call $V$  an {\it $\FI\mathsf{[G]}$-module over $\mathbb{F}$} if the functor $V:\Mod_{\FI}\rightarrow \Mod_{\mathbb{Z}\mathsf{[G]}}$ factors through $\Mod_{\mathbb{F}\mathsf{[G]}}\rightarrow\Mod_{\mathbb{Z}\mathsf{[G]}}$.
\end{definition}



Let $V$ be an $\FI\mathsf{[G]}$-module and consider a path connected space $X$ with fundamental group $G$. For each integer $p\geq 0$, the $pth$-cohomology $H^p(X;\_\_)$ of $X$  is  a  covariant functor from the category  $\Mod_{\FI\mathsf{[G]}}$ to the category $\Mod_{\FI}$. We now see how the bounds on the stable and local degree of $V$ provide  bounds for the $\FI\mathsf{[G]}$-module $H^p(X;\_\_)$.





\begin{prop}[Cohomology with coefficients in a $\FI\mathsf{[G]}$-module with finite $prd$]\label{LEMMA} 
Let $G$ be the fundamental group of a connected CW complex $X$  and let $V$ be  an $\FI\mathsf{[G]}$-module presented  in finite degrees with $\delta(V)\leq D$ and $h^{\text{max}}(V)\leq \eta$. Then for every $p\geq 0$, the $\FI$-module $H^p(X;V)$ is presented  in finite degrees with
$$\delta\big(H^p(X;V)\big)\leq D\text{\   \  and \  \  }h^{\text{max}}\big(H^p(X;V)\big)\leq \max(2D-2,\eta).$$
Furthermore, if $X$ has finitely many cells in each dimension and $V$ is a finitely generated $\FI\mathsf{[G]}$-module, then  $H^p(X;V)$ is  a finitely generated  $\FI$-module.
\end{prop}


\noindent\begin{proof}

Given that $G=\pi_1(X)$, the universal cover $\tilde{X}$ of $X$  has a $G$-equivariant cellular chain complex $C_*(\tilde{X})$.  For each $p\geq 0$, let $\mathcal{A}_p$ be the set of representatives of $G$-orbits of the $p$-cells. The group of $p$-chains $C_p(\tilde{X})$  is a free $G$-module with a generator for each element in $\mathcal{A}_p$.

To obtain the $\FI$-module $H^p(X;V)$  we consider the complex of $\FI$-modules
$$C^{p-1}(X,V)\xrightarrow[]{\delta_{p-1}} C^{p}(X,V) \xrightarrow[]{\delta_{p}} C^{p+1}(X,V)$$ 
where $C^p(X;V)$  is a direct product of $\FI\mathsf{[G]}$-modules $\prod_{\alpha\in\mathcal{A}_p} V$  given by 
$$C^p(X;V)_n:=\Hom_G(C_p(\tilde{X}),V_n)=\Hom_{\mathbb{Z}[G]}\big(\bigoplus_{\alpha\in\mathcal{A}_p}\mathbb{Z}[G],V_n\big)\cong\prod_{\alpha\in\mathcal{A}_p}\Hom_{\mathbb{Z}[G]}(\mathbb{Z}[G],V_n)\cong\prod_{\alpha\in\mathcal{A}_p} V_n.$$
In the category  $\Mod_{\FI\mathsf{[G]}}$ direct products preserve exactness (since in $\Mod_{\mathbb{Z}\mathsf{[G]}}$  direct products preserve exactness) and 
$\text{Ind}_{S_{n-1}}^{S_n} \prod_{\alpha\in\mathcal{A}_p} V\cong\prod_{\alpha\in\mathcal{A}_p}\text{Ind}_{S_{n-1}}^{S_n} V$  (see for example \cite[Prop. 4.4]{LAM}). It follows  from the definitions of stable degree and local degree that 
 $\delta(\prod_{\alpha\in\mathcal{A}_p} V)=\delta(V)$  and  $h^{\text{max}}(\prod_{\alpha\in\mathcal{A}_p} V)=h^{\text{max}}(V)$ (in the case $\mathcal{A}_p$ is finite, this is just an immediate consequence from Proposition \ref{KERCOKER} (1)). Therefore,  for every $p\geq 0$
 $$\delta\big(C^p(X;V)\big)\leq D\text{\hspace{0.5cm}  and \hspace{0.5cm}}h^{\text{max}}\big(C^p(X;V)\big)\leq \eta.$$
Since $$H^p(X;V)=\text{coker}\big(C^{p-1}(X,V)\rightarrow \ker \big( \delta_p: C^{p}(X,V)\rightarrow  C^{p+1}(X,V)\big)\big),$$
the $\FI$-module $H^p(X;V)$ is presented in finite degrees and it follows from Proposition \ref{KERCOKER} (2) that
$$\delta\big(H^p(X;V)\big)\leq \delta\big( \ker \big( \delta_p: C^{p}(X,V)\rightarrow  C^{p+1}(X,V)\big))\leq \delta\big(C^{p+1}(X,V)\big)=\delta(V)\leq D$$
and
$$h^{\text{max}}\big(H^p(X;V)\big)\leq \max\big(2\delta(C^{p-1}(X,V))-2, h^{\text{max}}(C^{p-1}(X,V)), h^{\text{max}}(\ker ( \delta_p))\big)$$
$$\text{ \hspace{0.5cm}}\leq \max\big(2D-2, \eta,\max\big(2D-2,\eta,\eta)\big)= \max (2D-2,\eta).$$

Finally, if $X$ has finitely many cells in each dimension and $V$ is is a finitely generated $\FI\mathsf{[G]}$-module, then  $H^p(X;V)$ it is a subquotient of finitely generated $\FI$-modules $C^p(X;V)$.
\end{proof}

\begin{remark} If $X$ has finitely many cells in each dimension and $V$ is  a finitely generated $\FI\mathsf{[G]}$-module over $\mathbb{Q}$ with weight $\leq m$ and stability degree $N$, then  $H^p(X;V)$ has weight $\leq m$ and stability degree $N$ (\cite[Prop 5.1]{JimenezRollandMCGasFI}).
\end{remark}

\subsection{$\FI\mathsf{[G]}$-modules and spectral sequences}
Let $A$ be an abelian group.
Consider a functor from {$\FI^{op}$} to the category of group extensions with quotient $G$. Let us denote by $E,H: \FI^{op}\rightarrow\mathcal{G}rps$ the corresponding $\FI^{op}$ groups, such that the group extension associated to { $[n]$} is given by
$$ 1\rightarrow H_n\rightarrow E_n\rightarrow G\rightarrow 1.$$
The action of $G$ on the cohomology groups $H^q(H_n;A)$, gives $H^q(H;A)$  the structure of an $\FI\mathsf{[G]}$-module, for all $q\geq 0$.
 
By considering the Lyndon--Hochschild--Serre spectral sequences associated to these extensions, we obtain a first quadrant spectral sequence of $\FI$-modules converging to the graded $\FI$-module $H^*(E;A)$, with $E_2$-term
$$E_2^{p,q}=H^p\big(G;H^q(H;A)\big).$$

More generally, let $X$ be a connected CW-complex, %
 $x\in X$ and  suppose that the fundamental group $\pi_1(X,x)$ is $G$.  Consider an $\FI^{op}$-fibration over $X$: a functor from { $\FI^{op}$} to the category $\Fib(X)$ of fibrations over $X$. Let $E: \FI^{op}\rightarrow\mathcal{T}op$ be the $\FI^{op}$-space   of total spaces and $H$ the $\FI^{op}$-space of  fibers over the basepoint $x$. 
 The functoriality of the Leray-Serre spectral sequence implies that we have a spectral sequence of $\FI$-modules   (see  \cite[Proposition 3.10]{MILLERKUPERS} for details) $E^{p,q}_*=E^{p,q}_*\big(E\rightarrow X\big)$ converging to the graded $\FI$-module $H^*(E;A)$.

  The action of the fundamental group $G$ on the cohomology of the fiber gives $H^q(H;A)$ the structure of an $\FI\mathsf{[G]}$-module, for all $q\geq 0$. Then, the $E_2$-page of this spectral sequence is given by the $\FI$-modules
 $$E_2^{p,q}=H^p\big(X;H^q(H;A)\big).$$ 






For the families of groups and spaces  that we are interested in this paper, this setting will give us cohomologically graded first quadrant spectral sequences of $\FI$-modules converging to their integral cohomology. 

\section{Integral cohomology of pure mapping class groups of surfaces}\label{MCG}

 In this section we study the $\FI$-module structure of the integral cohomology of  the pure mapping class groups of orientable and non-orientable surfaces. We recall how the Birman exact sequence and certain fibrations allows to relate the cohomology of pure mapping class groups with the cohomology of configuration spaces. The spectral sequence arguments from \cite{CMNR} together with  their the bounds on stable degree and local degree for cohomology of configuration spaces (see Lemma \ref{CONF} below) and our Proposition \ref{LEMMA} will allow us to obtain the corresponding bounds for pure mapping class groups in Theorem \ref{MAINMCG}.
 
 \subsection{Configuration spaces} 
For a topological space $M$, following the notation in \cite{CMNR}, we denote by  $\PConf(M)$  the $\FI^{op}$-space that sends the  set $S$ to the space of embeddings of $S$ into $M$. For a given inclusion $f:S\hookrightarrow T$ in $\Hom_{\FI}(S, T)$ the corresponding restriction $f_*: \PConf_S(M)\rightarrow \PConf_T(M)$ is given by precomposition. In particular, for the set $[n] = \{1, . . . , n\}$ we obtain the configuration space $\PConf_n(M)=$Emb$([n], M)$  of ordered $n$-tuples of distinct points on $M$. By taking  integral cohomology we obtain an $\FI$-module $H^k(\PConf(M);\mathbb{Z})$.  The following Lemma \ref{CONF} is one of the main applications in \cite{CMNR}  of  Proposition \ref{SPECTRAL} and is a key ingredient in the proof of Theorems \ref{MAINMCG} and \ref{DIFFEO}.

 \begin{lemma}[Theorem 4.3 \cite{CMNR} Linear ranges for configuration spaces]\label{CONF}Let $M$ be a connected manifold of dimension $d\geq 2$, and set

\begin{multicols}{2} 
$ \mu= 
    \begin{cases}
       2& \text{if  }d=2\\
       1& \text{if  }d\geq 3
   \end{cases}$
  
   $\lambda= 
    \begin{cases}
       0& \text{if  $M$  is non-orientable}\\
       1& \text{if $M$ is orientable}
   \end{cases}$
   \end{multicols}
\noindent Let $A$ be an abelian group. Then we have:
 \begin{itemize}
\item [(1)] $\delta\big(H^q(\PConf(M);A)\big)\leq \mu q,$
\item [(2)] $h^{\text{max}}\big(H^q(\PConf(M);A)\big)\leq \max(-1,4\mu q-2\mu\lambda -2),$ 
\item [(3)] $t_0\big(H^q(\PConf(M;A)\big)\leq \max(\mu q,5\mu q-2\mu\lambda -1),$ 
\item [(4)] $t_1\big(H^q(\PConf(M);A)\big)\leq \max(\mu q,9\mu q-4\mu\lambda -2).$ 
 \end{itemize}
 \end{lemma}
 
\subsection{Pure mapping class groups of surfaces}
 
 Let $M$ be a connected, smooth manifold, and $\mathfrak{p}\in \PConf_n(\mathring{M})$, in other words $\big(\mathfrak{p}(1),\mathfrak{p}(2),\ldots,\mathfrak{p}(n)\big)$  is an ordered  configuration of $n$ points  in the interior $\mathring{M}$ of $M$.   
 We denote by $\Diff(M)$ the group of diffeomorphisms of $M$. By abusing notation, we use $\Diff (M )$ instead of  $\Diff(M \text{ rel } \partial M)$  if $\partial M\neq\varnothing$ and if  $M$ is orientable, we can restrict to orientation-preserving diffeomorphisms.  
 
 Let $\Diff^{\mathfrak{p}} (M)=\Diff(M, P)$ be   the subgroup of $\Diff (M)$ of diffeomorphisms that leave invariant the set $P=\mathfrak{p}([n])$ and $\PDiff^{\mathfrak{p}} (M)=\PDiff(M,P)$ is the subgroup of $\Diff (M)$  that consists of the diffeomorphims that fix each  point in $P$.  Since $M$ is connected, if $\mathfrak{p},\mathfrak{q}\in \PConf_n(\mathring{M})$, then $\Diff^{\mathfrak{p}} (M)\approx\Diff^{\mathfrak{q}} (M)$ and $\PDiff^{\mathfrak{p}} (M)\approx \PDiff^{\mathfrak{q}} (M)$. We denote them by $\Diff^n (M)$ and $\PDiff^n (M)$, respectively.

  The mapping group with marked points is defined as $\Mod^n(M)=\pi_0\big(\Diff^{n} (M)\big)$ and the pure mapping class group is $\PMod^n(M)=\pi_0\big(\PDiff^{n} (M)\big)$.  When no marked points are considered, we just write $\Mod(M)$. Following the notation in \cite{JimenezRollandRepStability}, we consider the $\FI^{op}$-group $\PMod^{\bullet}(M)$  given by $[n]\mapsto \PMod^n(M)$. An inclusion $f:[m]\hookrightarrow[n]$ induces a restriction $\PDiff^{\mathfrak{p}} (M)\rightarrow \PDiff^{\mathfrak{p}\circ f} (M)$, which induces the homomorphism $f^*:\PMod^n(M)\rightarrow \PMod^m(M)$. By taking  integral cohomology we obtain an $\FI$-module $H^k(\PMod^{\bullet}(M);\mathbb{Z})$, for $k\geq 0$.
 
 In what follows we restrict ourselves to the case when $M=\Sigma$ is a surface. If $\Sigma$ is orientable we consider orientation-preserving diffeomorphisms.

\subsection{The Birman exact sequence}
Let $2g+r>2$, there exists a short exact sequence  

\begin{equation}\label{BIR1}
1\rightarrow \pi_1(\PConf_n(\Sigma_g^{r}))\rightarrow \PMod^n(\Sigma_{g,r})\rightarrow \Mod(\Sigma_{g,r})\rightarrow 1.
\end{equation}

This is the {\it Birman exact sequence} for the mapping class group of an orientable surface introduced by Birman  (\cite{BIRMANSS}  \cite{BIRMAN}). 

There also exists a Birman exact sequence for the mapping class group of a non-orientable surface (see for example  \cite[Theorem 2.1]{KOR} and  \cite[Theorem 5.3] {HamNONORIENT})
when  $g\geq3$ 
\begin{equation}\label{BIR2}
1\rightarrow \pi_1(\PConf_n(N_{g}^r))\rightarrow \PMod^n(N_{g,r})\rightarrow \Mod(N_{g,r})\rightarrow 1.
\end{equation}

Therefore, for $\Sigma$ a surface as before (orientable or non-orientable) we consider the functor from $\FI^{op}$ to the category of group extensions of $G=\Mod( \Sigma)$ given as follows. The group extension associated to $[n]$ is
 
 \begin{equation}\label{BIR}
1\rightarrow \pi_1(\PConf_n(\mathring{\Sigma}))\rightarrow \PMod^n(\Sigma)\rightarrow \Mod(\Sigma)\rightarrow 1,
\end{equation}
The action of $\Mod(\Sigma)$ on  $H^q\big(\pi_1(\PConf_n(\mathring{\Sigma});A\big)=H^q\big(\PConf_n(\mathring{\Sigma});A\big)$, gives $H^q\big(\PConf(\mathring{\Sigma});A\big)$  the structure of an  $\FI\mathsf{[G]}$-module, for any abelian group $A$.
To see that this association is indeed functorial we refer the reader to \cite[Section 5] {JimenezRollandRepStability}.

\subsection{Low genus surfaces}


For low genus surfaces, there is not a Birman exact sequence. Fortunately the following fibrations will allow us below, in Theorem \ref{MAINMCG}, to use the spectral sequence argument from Proposition \ref{SPECTRAL}  to obtain finite presentation for the corresponding $\FI$-modules.


\begin{prop}\label{FIBRATIONS} There exists  fibrations  
\begin{itemize}
\item [a)] {\bf For the sphere (\cite[Theorem A]{SMALE}, \cite[Lemma 2.1]{COHEN}):}

\begin{equation*}
\PConf_n(S^2)\rightarrow ESO(3)\times_{SO(3)}\PConf_n(S^2)\rightarrow BSO(3) \end{equation*}

where $ESO(3)\times_{SO(3)}\PConf_n(S^2)$ is  a $K(\PMod^n(S^2),1)$ for all $n\geq 3$

\item[(b)]{\bf For the punctured torus (\cite[Section 2]{TORUS}):}  

\begin{equation*}
\PConf_n(\Sigma_1^1)\rightarrow ESL(2,\mathbb{Z})\times_{SL(2,\mathbb{Z})}\PConf_n(\Sigma_1^1)\rightarrow BSL(2,\mathbb{Z}) \end{equation*}

where 
$ESL(2,\mathbb{Z})\times_{SL(2,\mathbb{Z})}\PConf_n(\Sigma_1^1)$ is  a $K(\PMod^{n}(\Sigma_1^1),1)$ for all $n\geq 1$.

\item[(c)]{\bf For the  projective plane(\cite{EE},\cite[Corollary 2.6]{WANG}):}

\begin{equation*}
\PConf_n(\mathbb{R}P^2)\rightarrow ESO(3)\times_{SO(3)}\PConf_n(\mathbb{R}P^2)\rightarrow BSO(3)
\end{equation*}

where
$ESO(3)\times_{SO(3)}\PConf_n(\mathbb{R}P^2)$ is a $K(\PMod^n(\mathbb{R}P^2),1)$ for $n\geq 2$.

\item[(d)] {\bf For the Klein bottle( \cite[Theorems 6.1 \& 7.1]{KLEINBUNDLES}):} 
\begin{equation*}
\PConf_n(\mathbb{K})\rightarrow E\left(\mathbb{Z}_2\times O(2)\right)\times_{\mathbb{Z}_2\times O(2)}\PConf_n(\mathbb{K})\rightarrow B\left(\mathbb{Z}_2\times O(2)\right)
\end{equation*}

where $E\left(\mathbb{Z}_2\times O(2)\right)\times_{\mathbb{Z}_2\times O(2)}\PConf_n(\mathbb{K})$ is  a $K(\PMod^n(\mathbb{K}),1)$ for all $n\geq 1$

\end{itemize}
\end{prop}
\begin{proof}
\begin{itemize}
\item[(a)] It is a classical theorem of Smale (\cite[Theorem A]{SMALE}) that the inclusion $SO(3)\hookrightarrow \Diff^+(S^2)$ is a homotopy equivalence.  On the other hand, it is known that $ESO(3)\times_{SO(3)}\PConf_n(S^2)$ is  a $K(\pi,1)$ for $n\geq 3$ (see for instance the proof  of \cite[Lemma 2.1]{COHEN}).
\item[(b)]  Since both   $\PConf_n(\Sigma_1^1)$ and  $ BSL(2,\mathbb{Z})$ are $K(\pi,1)$'s,  it follows that  $ESL(2,\mathbb{Z})\times_{SL(2,\mathbb{Z})}\PConf_n(\Sigma_1^1)$  is a $K(\pi,1)$.  Moreover 
$$\pi_1\big(ESL(2,\mathbb{Z})\times_{SL(2,\mathbb{Z})}\PConf_n(\Sigma_1^1)\big)\cong \PMod^{n}(\Sigma_1^1)$$
See \cite[Section 2]{TORUS} for further discussion of how to use this types of fibrations to compute rational cohomology of the pointed mapping class group of the torus.

\item[(c)] It was proved in \cite{EE} that the inclusion $SO(3)\hookrightarrow \Diff(\mathbb{R}P^2)$  is a homotopy equivalence. From \cite[Cor 2.6]{WANG}, 
$ESO(3)\times_{SO(3)}\PConf_n(\mathbb{R}P^2)$ is a $K(\pi, 1)$   for $n\geq 2$.
\item[(d)] \cite[Theorem 6.1]{KLEINBUNDLES}) shows that the inclusion of $\mathbb{Z}_2\times O(2)\hookrightarrow \Diff(\mathbb{K})$ is a homotopy equivalence and in the proof of \cite[Theorem 7.1]{KLEINBUNDLES} the authors show that the Borel construction $E\left(\mathbb{Z}_2\times O(2)\right)\times_{\mathbb{Z}_2\times O(2)}\PConf_n(\mathbb{K})$ is a $K(\pi,1)$, for $n\geq 1$.
\end{itemize}
To obtain the corresponding fundamental groups in each case recall that, for a closed  manifold $M$, the Borel construction $$E\Diff(M)\times_{\Diff(M)}\PConf_n(M)\approx E\Diff(M)\times_{\Diff(M)}\frac{\Diff(M)}{\PDiff^n(M)}=\frac{E\Diff(M)}{\PDiff^n(M)}\approx B\PDiff^n(M)$$
which has fundamental group
$$\pi_1\big( E\Diff(M)\times_{\Diff(M)}\PConf_n(M)\big)=\pi_1\big( B\PDiff^n(M)\big)=\PMod^n(M),$$
where we consider   $\Diff^+(M)$ instead of $\Diff(M)$  if $M$ is orientable.
\end{proof}

Therefore, from Proposition \ref{FIBRATIONS} , we have  fibrations   over some space $X^{\Sigma}$
\begin{equation}\label{FIBLOW}
\PConf_n(\Sigma)\rightarrow E_n^{\Sigma}\rightarrow X^{\Sigma}
\end{equation}
with fiber $\PConf_n(\Sigma)$, for $\Sigma=S^2,\Sigma_1^1,\mathbb{R}P^2,\mathbb{K}$.
In each cases, the space $X^{\Sigma}$  has the homotopy type of a connected CW-complex with finitely many cells  in each dimension, and we obtain an $\FI^{op}$-fibration over $X^{\Sigma}$ that can be used to obtain information about  the integral cohomology of the corresponding pure mapping class group $\PMod^n(\Sigma)$. The action of $G=\pi_1(X^{\Sigma})$ on the fiber gives $H^q\big(\PConf(\Sigma);A\big)$ the structure of an  $\FI\mathsf{[G]}$-module, for any abelian group $A$.

\subsection{Linear bounds for the presentation degree of $H^k(\PMod^{\bullet}(\Sigma))$}


We  now have all the ingredients to obtain linear bounds on the presentation degree of the $\FI$-modules $H^k(\PMod^{\bullet}(\Sigma);A)$, for any abelian group $A$.

\begin{theo}[Linear bounds for integral cohomology]\label{MAINMCG} Let  $\Sigma$ be a surface such that:
\begin{itemize}
 \item[] $\Sigma=S^2, \Sigma_1^1$  or  $\Sigma=\Sigma_{g,r}$ with $2g+r>2$, if  the surface is orientable,  or 
 
 \item[] $\Sigma=\mathbb{R}P^2, \mathbb{K}$ or $\Sigma=N_{g,r}$  with $g\geq 3$  and $r\geq 0$ in the non-orientable case. 
 \end{itemize} 
Let $A$ be any abelian group. Then for any $k\geq 0$ the cohomology $H^k\big( \PMod^\bullet(\Sigma);A\big)$ is a finitely presented $\FI$-module and
\begin{itemize}
\item[(1)] $\delta\big(H^k\big( \PMod^\bullet(\Sigma);A\big)\big)\leq 2k$,
\item[(2)]  $h^{\text{max}}\big(H^k\big( \PMod^\bullet(\Sigma);A\big)\big)\leq \max(-1,16k-4\lambda-2)$,
\item[(3)]  $t_0\big(H^k\big( \PMod^\bullet(\Sigma);A\big)\big)\leq \max(0,18k-4\lambda-1)$,
\item[(4)] $t_1\big(H^k\big( \PMod^\bullet(\Sigma);A\big)\big)\leq \max(0,34k-8\lambda-2)$,
\end{itemize}
where $\lambda=1$ if $\Sigma$ is orientable and $\lambda=0$ if $\Sigma$ is non-orientable.
\end{theo}

\begin{proof}
For $\Sigma$ an orientable surface with $2g+r>2$ or a non-orientable surface of genus $g\geq 3$ we consider the Birman exact sequence (\ref{BIR}). Associated to these group extensions  we have cohomologically graded first quadrant Hochschild-Serre spectral sequences  of $\FI$-modules  with $E_2-$term
 $$E_2^{p,q}= H^{p}\big(\Mod(\Sigma);H^q(\PConf(\mathring{\Sigma});A)\big),$$
 converging to $H^{p+q}(\PMod^\bullet\big(\Sigma);A\big).$
 
 For $\Sigma=S^2,\Sigma_1^1,\mathbb{R}P^2,\mathbb{K}$,  we can use the fibrations (\ref{FIBLOW}) from Proposition \ref{FIBRATIONS}  to obtain graded first quadrant Serre spectral sequences  of $\FI$-modules  with $E_2-$term

$$E_2^{p,q}= H^{p}\big(X^{\Sigma};H^q(\PConf(\Sigma);A)\big),$$
 converging to $H^{p+q}(\PMod^\bullet\big(\Sigma);A\big).$
 
 In all cases, from  the linear bounds from Lemma \ref{CONF} and Proposition \ref{LEMMA}, we obtain
 $\delta(E_2^{p,q})\leq 2q$  and $h^{\text{max}}(E_2^{p,q})\leq \max(-1,8q-4\lambda-2)$, where $\lambda=1$ if $\Sigma$ is orientable and $\lambda=0$ if $\Sigma$ is non-orientable.
Then $$D_k=\max_{p+q=k} \delta(E_2^{p,q})\leq 2k\text{;  }\eta_k=\max_{p+q=k} h^{\text{max}}(E_d^{p,q})\leq \max(-1,8k-4\lambda-2).$$
Therefore, from Proposition \ref{SPECTRAL} we obtain

$$\delta\big(H^k\big( \PMod^\bullet(\Sigma);A\big)\big)\leq 2k\text{; }$$
$$h^{\text{max}}\big(H^k\big( \PMod^\bullet(\Sigma);A\big)\big)\leq  \max \big(\max_{\ell\leq 2k} \eta_\ell, \max_{\ell\leq 2k-1} (4\ell-2)\big)$$ $$\leq  \max \big(\max(-1,8(2k)-4\lambda-2), 4(2k-1)-2\big)\leq  \max(-1,16k-4\lambda-2),$$
which are the bounds in (1) and (2). From Proposition \ref{BOUNDS} we obtain 

$$t_0\big(H^k\big( \PMod^\bullet(\Sigma);A\big)\big)\leq 2k+\max(-1,16k-4\lambda-2)+1=\max(0,18k-4\lambda-1)$$
and
$$t_1\big(H^k\big( \PMod^\bullet(\Sigma);A\big)\big)\leq 2k+2\big(\max(-1,16k-4\lambda-2)\big)+2=\max(0,34k-8\lambda-2).$$

Furthermore, since the groups  $\Mod(\Sigma)$ are of type $FP_\infty$  (see for example \cite[Section 6]{IVANOV}) and the spaces $X^{\Sigma}$ have the homotopy type of a CW-complex with finitely many cells in each dimension, it follows from Proposition \ref{LEMMA} that $H^k\big(\PMod^\bullet(\Sigma);A\big)$ are finitely generated $\FI$-modules (hence finitely presented).
\end{proof}\medskip

\noindent{\bf Surfaces with boundary.}  
If $\Sigma$ has nonempty boundary, then $H^k(\PMod^{\bullet}(\Sigma);A)$ has an $\FI\#$-module structure for $k\geq 0$ (see for example \cite[Prop 6.3]{JimenezRollandMCGasFI}). Therefore, by Proposition \ref{FISHARP}, we obtain better  bounds for the presentation degree. 

\begin{theo}[Surfaces with boundary]\label{BDRY}
Let  $\Sigma=\Sigma_{g,r}$ with $g\geq 1$ and $r\geq 1$ or $\Sigma=N_{g,r}$ with $g\geq 3$ and $r\geq 1$.
Then for any abelian group $A$ and $k\geq 0$ the cohomology $H^k\big( \PMod^\bullet(\Sigma);A\big)$ is a finitely presented $\FI\#$-module with  local  degree $=-1$ and with stable degree, generation degree and presentation degree bounded above by $2k$.
\end{theo}


\subsection{Cohomology of hyperelliptic mapping class groups}\label{HYPERSEC}

For a surface $\Sigma_g$ the {\it hyperelliptic diffeomorphism} is the order two map with $2g + 2$ fixed points and which acts as a rotation of $\pi$ around a central axis of $\Sigma_g$. We denote the mapping class of this diffeomorphism in $\Mod(\Sigma_g)$ by $\iota$.

\begin{definition} The {\it hyperelliptic mapping class group} $\Delta_g$  is the normalizer of  $\langle\iota\rangle\cong \mathbb{Z}_2$ in $\Mod(\Sigma_g)$. The group $\Delta_g^n$, the  {\it hyperelliptic mapping class group with marked points},  is defined as the preimage of $\Delta_g$ under the forgetful map $\PMod^n(\Sigma_g)\rightarrow\Mod(\Sigma_g)$. 
\end{definition}

For $g=1,2$, we have that $\Delta_g=\Mod(\Sigma_g)$, but for $g \geq 3$, $\Delta_g$  is neither normal nor of finite index in $\Mod(\Sigma_g)$. Moreover, there exists a non split central extension (see \cite{BIRHILDEN})
\begin{equation}\label{HYPER}
1\rightarrow \mathbb{Z}/2\mathbb{Z}\rightarrow\Delta_g\rightarrow \Mod^{2g+2}(\Sigma_{0})\rightarrow 1.\end{equation}

  The cohomology  ring of $\Delta_g$ has  been studied before, see for example \cite{BCP} and \cite{COHEN}.  In particular, the map $\Delta_g\rightarrow  \Mod^{2g+2}(\Sigma_{0})$ gives a cohomology isomorphism with coefficients in a ring $R$ containing $1/2$.

{


For $g\geq 3$, by restricting  the short exact sequence (\ref{BIR1}), we obtain a Birman exact sequence for hyperelliptic mapping class groups : 

\begin{equation}\label{BIRHYPER}
1\rightarrow \pi_1(\PConf_n(\Sigma_g))\rightarrow \Delta^n_{g}\rightarrow \Delta_{g}\rightarrow 1.
\end{equation}

We can consider the $\FI^{op}$-group $\Delta^{\bullet}_{g}$ given by $[n]\mapsto \Delta_g^n$ and for each inclusion $f:[m]\hookrightarrow[n]$, we restrict the homomorphism $f^*:\PMod^n(\Sigma_g)\rightarrow\PMod^m(\Sigma_g)$ to  $f^*:\Delta_g^n\rightarrow\Delta^m_g$. Therefore, for any abelian group $A$ and $k\geq 0$, we obtain  $\FI$-modules $H^k\big( \Delta^{\bullet}_{g};A\big)$.  From (\ref{HYPER}) it follows that  $\Delta_g$ is of type $FP_{\infty}$.
Using the sequence (\ref{BIRHYPER}) and Lemma \ref{CONF} in the setting of the  proof of Theorem \ref{MAINMCG}, it follows that this $\FI$-modules  are finitely presented. Moreover, with $\mathbb{Q}$- coefficients, the proof of \cite[Theorem 6.1]{JimenezRollandMCGasFI}  will give the same upper bounds for the  weight and stability degree.

\begin{theo}\label{HYPERMCG} 
Let $g\geq 3$ and $k\geq 0$ and $A$ be any abelian group, then the cohomology $H^k\big( \Delta^{\bullet}_{g};A\big)$ is a finitely presented $\FI$-module with upper bounds for the stable, local, generation and presentation degrees as in Theorem \ref{MAINMCG} (with $\lambda=1$).
Furthermore, over $\mathbb{Q}$ the  $\FI$-module $H^k\big( \Delta^{\bullet}_{g};\mathbb{Q}\big)$ has  weight$\leq 2k  $ and stability degree $\leq 4k $ and the sequence $\{H^k\big( \Delta^{n}_{g};\mathbb{Q}\big)\}_n$ is representation stable for $n\geq 6k$.
\end{theo}


\section{Integral cohomology of classifying spaces for groups of  diffeomorphisms}

 Let $M$ be a connected and compact, smooth manifold, orientable or non-orientable, of dimension $d\geq 3$. We consider  $\BPDiff^\bullet(M):\FI^{op}\rightarrow \mathcal{T}op$ to be the $\FI^{op}$-space given by  $[n]\mapsto \BPDiff^n(M)$, where $\BPDiff^n(M)$ is the classifying space of the group $\PDiff^n(M)$. Again, if $M$ is orientable, we can restrict to orientation-preserving diffeomorphisms.  There is a fiber bundle
\begin{equation}\label{FIB}
\BPDiff^n(M)\rightarrow \BDiff(M)
\end{equation}
\noindent where the fiber is  $\Diff(M)/\PDiff^n(M) \approx \PConf_n(\mathring{M})$, the configuration space of $n$ ordered points in the interior of $M$. This gives us a functor from {$\FI^{op}$} to the category {$\Fib\big(\BDiff(M)\big)$.


\begin{theo}\label{DIFFEO} Let $M$ be a connected real manifold (orientable or non-orientable) of dimension $d\geq 3$. 
Let $A$ be any abelian group, then for  $k\geq 0$, the  $\FI$-module $ H^k(B\PDiff^\bullet M;A)$ is presented in finite degrees and
\begin{itemize}
\item[(1)] $\delta\big(H^k\big( B\PDiff^\bullet M;A\big)\big)\leq k$,
\item[(2)]$h^{\text{max}}\big(H^k\big(  B\PDiff^\bullet M;A\big)\big)\leq \max(-1,8k-2\lambda-2)$,
\item[(3)] $t_0\big(H^k\big(  B\PDiff^\bullet M;A\big)\big)\leq\max(0,9k-2\lambda-1)$,
\item[(4)] $t_1\big(H^k\big( B\PDiff^\bullet M;A\big)\big)\leq \max(0,17k-4\lambda-2)$.
\end{itemize}
Furthermore, if $B\Diff(M)$ has the homotopy type of a CW-complex with finitely many cells in each dimension, then $ H^k(B\PDiff^\bullet M;A)$ is finitely generated. 
\end{theo}
\begin{proof}
 Associated to the fibrations (\ref{FIB}) we have  a cohomologically graded first quadrant Leray-Serre spectral sequence  of $\FI$-modules converging to $H^{p+q}(B\PDiff^\bullet M;A)$ with $E_2-$term:
 $$E_2^{p,q}= H^{p}\big(\BDiff(M);H^q(\PConf(M);A)\big).$$
  From  the linear bounds from Lemma \ref{CONF} and Proposition \ref{LEMMA}, we obtain
 $\delta(E_2^{p,q})\leq q$  and $h^{\text{max}}(E_2^{p,q})\leq \max(-1,4q-2\lambda-2)$, where $\lambda=1$ if $M$ is orientable and $\lambda=0$ if $M$ is non-orientable.
Then $$D_k=\max_{p+q=k} \delta(E_2^{p,q})\leq k\text{;  }\eta_k=\max_{p+q=k} h^{\text{max}}(E_d^{p,q})\leq \max(-1,4k-2\lambda-2).$$
Therefore, from Proposition \ref{SPECTRAL} we obtain
$$\delta\big(H^k(B\PDiff^\bullet M;A)\big)\leq k\text{; }$$
$$h^{\text{max}}\big(H^k(B\PDiff^\bullet M;\mathbb{Z})\big)\leq  \max \big(\max_{\ell\leq 2k} \eta_\ell, \max_{\ell\leq 2k-1} (2\ell-2)\big)$$ $$\leq  \max \big(\max(-1,4(2k)-2\lambda-2), 2(2k-1)-2\big)\leq  \max(-1,8k-2\lambda-2),$$
which are the bounds in (1) and (2). From Proposition \ref{BOUNDS} we obtain 
$$t_0\big(H^k\big( \PMod^\bullet(\Sigma);A\big)\big)\leq k+\max(-1,8k-2\lambda-2)+1=\max(0,9k-2\lambda-1)$$
and
$$t_1\big(H^k\big( \PMod^\bullet(\Sigma);A\big)\big)\leq k+2\big(\max(-1,8k-2\lambda-2)\big)+2=\max(0,17k-4\lambda-2).$$

Finally, if $B\Diff(M)$ has the homotopy type of a CW-complex with finitely many cells in each dimension, then we get finite generation from Proposition \ref{LEMMA} \end{proof}

}

\begin{remark}  Theorem \ref{DIFFEO} is an integral version of \cite[Theorem 6.6]{JimenezRollandMCGasFI} that includes de case of $M$ being non-orientable.    In the statement of \cite[Theorem 6.6]{JimenezRollandMCGasFI} the hypothesis of $M$ being orientable was not explicitly stated, but was necessary at that time since   \cite[Theorem 4]{CEF} was only proven for orientable manifolds.
\end{remark}

From   Theorem \ref{DIFFEO} and Propositions \ref{POLY} and \ref{INDUCT} we obtain  Theorem \ref{DIFFEOCOR}.


\bibliographystyle{amsalpha}
\bibliography{referFI}

\begin{small}
\noindent Instituto de Matem\'aticas, Universidad Nacional Aut\'onoma de M\'exico \\
Oaxaca de Ju\'arez, Oaxaca, M\'exico 68000 \\ 
E-mail: \texttt{\href{mailto:rita@im.unam.mx}{rita@im.unam.mx}}
\\
\end{small}
 
 \end{document}